\documentclass[12pt,a4paper,oneside,onecolumn]{article}
\usepackage{mathrsfs}
\usepackage{amsfonts}
\usepackage{latexsym}
\usepackage{amsmath,amsthm}
\usepackage{epsf}
\usepackage{graphicx}
\usepackage{indentfirst}
\usepackage{cite}
\usepackage{color}
\usepackage{booktabs}
\usepackage{multirow}
\def\thesection{\arabic{section}}
\def\thesubsection{\arabic{section}.\arabic{subsection}}
\usepackage{caption}
\renewcommand{\figurename}{Fig.}
\renewcommand{\baselinestretch}{1.2}
\renewcommand{\arraystretch}{1}
\newcommand{\double}{\baselineskip 1 \baselineskip}

\theoremstyle{plain}
\newtheorem{thm}{\bf Theorem}[section]
\newtheorem{con}[thm]{\bf Construction}

\newtheorem{lem}[thm]{\bf Lemma}

\theoremstyle{definition}

\newtheorem{exam}{\bf Example}

\theoremstyle{remark}
\newtheorem{rmk}{\bf Remark}

\newcommand{\abs}[1]{\lvert#1\rvert}
\newcommand{\absb}[1]{\Bigl|#1\Bigr|}
\newcommand{\norm}[1]{\lVert#1\rVert}
\newcommand{\normb}[1]{\Big\lVert#1\Big\rVert}
\newcommand{\nnorm}[1]{| \mspace{-2mu} |\mspace{-2mu}|#1| \mspace{-2mu}
|\mspace{-2mu}|}
\newcommand{\spd}[2]{\protect\langle #1,#2\protect\rangle}
\title{\bf  New results on large sets of orthogonal arrays  and orthogonal arrays}
\author{Guangzhou Chen$^1$, Xiaodong Niu$^{1,2}$, Jiufeng Shi$^1$\\
\small1. {\it Henan Engineering Laboratory for Big Data Statistical Analysis and Optimal Control, }\\
\small {\it School of mathematics and information science, Henan Normal University,}\\
\small {\it Xinxiang, 453007, P.R.China}\\
\small2. {\it School of mathematical sciences, Hebei Normal University, Shijiazhuang, 050024, P.R.China} \\
}
\date{}
\begin{document}
\maketitle
%abstract and keywords

\begin{center}
\center  {\bf Abstract}
\begin{minipage}{16cm}%\setlength{\parskip}{-7bp}
{%\rule{16cm}{0.5pt}
%\mbox{}\hspace{2.80in}  {\bf Abstract}

\vspace{0.3cm} \ \ \ \ Orthogonal array and a large set of
orthogonal arrays are important research objects in combinatorial design theory, and they are widely applied to statistics, computer science, coding theory and cryptography. In this paper, some new series of large sets of orthogonal arrays are given by direct construction, juxtaposition construction, Hadamard construction, finite field construction and difference matrix construction. Subsequently, many new infinite classes of orthogonal arrays are obtained by using these large sets of orthogonal arrays and Kronecker product.


%\rule{16cm}{0.5pt}
}
\end{minipage}
\end{center}

\vskip 0.5cm

 {  \begingroup\makeatletter  \let\@makefnmark\relax  \footnotetext{ Corresponding author: Guangzhou Chen (chenguangzhou0808@163.com)  }  \makeatother\endgroup}
% {  \begingroup\makeatletter  \let\@makefnmark\relax  \footnotetext{
% The authors are supported by the Natural Science Foundation of China
% (Grant Nos. 11871417, 11971104 )} \makeatother\endgroup}

\section{Introduction}

Orthogonal array is a combinatorial configuration introduced firstly by Rao \cite{C.R1} in 1947. It is one of important topics in combinatorial design theory \cite{Handbook,Stinson} and experimental design theory \cite{G.T,Fang}.

An \emph{orthogonal array} OA$(N, s_1^{k_1}s_2^{k_2}\cdots s_m^{k_m}, t)$ is an array of size $N \times k$, where $k=\sum_{i=1}^m k_i$ is the total number of factors, in which the first $k_1$ columns have symbols from $\mathbb{Z}_{s_1}=\{0, 1,\ldots, s_1-1\}$, the next $k_2$ columns have symbols from $\mathbb{Z}_{s_2}=\{0, 1,\ldots, s_2-1\}$, and so on, with the property that in any $N\times t$ sub-array, every possible $t$-tuple occurs an equal number of times as a row. If $s_1=s_2=\cdots=s_m=s$, then it is called \emph{symmetric}, denoted by OA$(N,k,s,t)$. Otherwise, it is called \emph{asymmetric} or \emph{mixed}. In an OA$(N, s_1^{k_1}s_2^{k_2}\cdots s_m^{k_m}, t)$, the rows are referred to as \emph{runs}, so $N$ is the number of runs of the OA, also called its \emph{size}. The coordinates of  $\mathbb{Z}_{s_1}^{k_1}\times\mathbb{Z}_{s_2}^{k_2}\times\cdots\times\mathbb{Z}_{s_m}^{k_m}$ are called \emph{factors}, so $k$ is the number of factors, also called \emph{degree}. The elements $\mathbb{Z}_{s_i}$ represent the \emph{levels} of the corresponding factor. The
parameter $t$ is called \emph{the strength}. In statistics, an OA$(N, s_1^{k_1}s_2^{k_2}\cdots s_m^{k_m}, t)$ is termed a \emph{full factorial design} if its rows consist of exactly those vectors of the Cartesian product
$\mathbb{Z}_{s_1}^{k_1}\times\mathbb{Z}_{s_2}^{k_2}\times\cdots\times\mathbb{Z}_{s_m}^{k_m}$. It is called \emph{trivial} if its rows consist of a number of copies of a full factorial design.

Orthogonal array is widely applied to many other fields, for example, in fractional factorial experiments (see \cite{Hedayat-OA,Cheng,Mukerjee,Gupta,G.T,Rao}), computer experiments (see \cite{Fang,lin10,Owen}) and survey sampling (see \cite{McCarthy,Wu}), computer science (see \cite{A.W}), coding theory (see \cite{B.J,Hedayat-OA}) and cryptography (see \cite{Carlet,J.D,D.P}). Recently, orthogonal arrays are used to construct entangled states in quantum information
(see \cite{G.D,X.Z,X.Z1}) and
wireless sensor network (see \cite{G.Q,S.Pang2,S.Pang,Du}).

Orthogonal array of strength 2 have been studied extensively in
the literature (see \cite{Handbook,Chen,Zhang2}) and a great deal of methods
and results can be found in the monograph (see \cite{Hedayat-OA}) and the
Handbook (see \cite{Handbook}). With the further study of orthogonal array,
there are a few constructions for orthogonal array of strength $3$. For a
review on these methods of construction, we refer the readers to   \cite{Agrawal,Gupta1,A.E.,Nguyen,Schoen10,Suen,Jiang13,Ji10,yin11,chen17,Zhang3}.
However, relatively less work on the constructions for orthogonal arrays of
strength greater than three are available. Recently,
Pang et al. (see \cite{Pang,Wang2}) used orthogonal partitions to obtain some
new infinite families of strength $t\geq 3$. Quite recently, Chen and
Niu \cite{chen23} used large sets of orthogonal arrays
and Kronecker product to construct some new orthogonal arrays with high strength.

An orthogonal array is \emph{simple} if it does not contain two identical
rows.  Let $\mathbb{Z}_v=\{0,1,2,\ldots,\\v-1\}$. A set $\mathcal{L}=\{A_1, A_2, \ldots ,A_M\}$ is said to be \emph{a large set of
orthogonal arrays}, denoted by LOA$(N,k,v,t)$ if each $A_i$, $i=1,2,\ldots,M$, is a
simple OA$(N,k,v,t)$ over $\mathbb{Z}_v$ and every possible $k$-tuple
over $\mathbb{Z}_v$ occurs in exactly one of the $M$ OAs. It is easy to compute that
$M=\frac{v^k}{N}$. Equivalently, the union of the $M$ simple OAs forms a trivial OA$(v^k, k, v,k)$. The notion of an LOA was proposed firstly by Stinson\cite{D.R}. It can be
used to construct zigzag function and resilient function \cite{D.R2,D.R},
which have important applications in computer science and cryptography. In addition, Zhang and Lei \cite{Z.Y} found that there is a close relationship between LOA and multimagic rectangles, and Zhang et al. \cite{Z.Y2} also introduced the concept of strong double large sets of orthogonal arrays to construct multimagic squares. The definition of an LOA$(N,k,v,t)$ was first generalized to a large set of orthogonal arrays with the mixed levels by Chen and Niu \cite{chen23}, which can be used to construct orthogonal arrays with high strength.

Let $k=\sum_{i=1}^mk_i$, a set $\mathcal{L}=\{A_1, A_2, \ldots ,A_M\}$ is said to be a large set of
OA$(N, s_1^{k_1}s_2^{k_2}\cdots s_m^{k_m}, t)$ over $S=\mathbb{Z}_{s_1}^{k_1}\times\cdots\times \mathbb{Z}_{s_m}^{k_m}$
if each $A_i$, $i=1,2,\ldots,M$, is a simple
OA$(N, s_1^{k_1}s_2^{k_2}\cdots s_m^{k_m}, t)$ over $S$ and every
possible $k$-tuple over $S$ occurs in exactly one of the $M$ simple
OAs in the set $\mathcal{L}$. Clearly, $M=\frac{\prod_{i=1}^ms_i^{k_i}}{N}$.
A large set of OA$(N, s_1^{k_1}s_2^{k_2}\cdots s_m^{k_m}, t)$ is denoted
by LOA$(N, s_1^{k_1}s_2^{k_2}\cdots s_m^{k_m}, t)$ for short.

In this article, we mainly focus on the constructions of large sets of orthogonal arrays and orthogonal arrays. The new results on large sets of orthogonal arrays are listed in \textbf{Tables 1-3} and \textbf{Table 5}, and the new results on orthogonal arrays are give by \textbf{Table 4}, Theorems \ref{v1+v3-2}-\ref{qtp43} and \textbf{Table 6} with some detailed parameters.

The rest of the article is organized as follows. In Section 2,
we give some new constructions of large sets of orthogonal arrays,
such as direct construction, juxtaposition construction, Hadamard
construction, finite field construction and difference matrix construction,
and then obtain some new series of large sets of orthogonal arrays by using
these constructions. In Section 3, many new infinite classes of orthogonal arrays
are obtained by using the large sets of orthogonal arrays and Kronecker product.
We conclude with some discussion of our work and possible future directions in Section 4.

\section{Constructions of large sets of orthogonal arrays}

In this section, we focus on constructions of large sets of orthogonal arrays
and obtain some new series of large sets of orthogonal arrays by using
these constructions.

\subsection{Large sets of orthogonal arrays with small runs}

In this subsection, we shall give some new large sets of orthogonal arrays with small runs
by using the following lemmas.

\begin{lem}[\cite{chen23}]\label{3-20}
If there exists an OA$(N, s_1 s_2 \cdots s_k, t)$, and there exists a subset
$\{i_1,i_2,\ldots,i_l\}$ $\subseteq\{1,2,\ldots,k\}$ such that $N=s_{i_1}s_{i_2}\cdots s_{i_l}$ and every possible $l$-tuple in the columns $i_1,i_2,\ldots,i_l$ occurs exactly once in the OA as a row, then there exists an LOA$(N, s_1 s_2 \cdots s_k, t)$.
\end{lem}

\begin{lem}[\cite{chen23}]\label{3-201}
If there exists an OA$(N, s_1 s_2 \cdots s_k, t)$ and there exists
a subset $\{i_1,i_2,\ldots,i_t\}\\\subseteq\{1,2,\ldots, k\}$ such that $N=s_{i_1}s_{i_2}\cdots s_{i_t}$, then there exists an
LOA$(N, s_1 s_2 \cdots s_k, t)$.
\end{lem}

Obviously, an OA$(N,k-i,s,t)$ can be obtained by deleting any $i$ columns from an OA$(N,k,\\s,t)$ directly, where $0\leq i\leq k-t$. Based on this fact, we have the following lemma.

\begin{lem}\label{3-202}
If there exists an OA$(N, s_1 s_2 \cdots s_k, t)$, and there exists a subset
$\{i_1,i_2,\ldots,i_l\}$ $\subseteq\{1,2,\ldots,k\}$ such that $N=s_{i_1}s_{i_2}\cdots s_{i_l}$ and every possible $l$-tuple in the columns $i_1,i_2,\ldots,i_l$ occurs exactly once in the OA as a row, then there exists an LOA$(N, s_{i_1} s_{i_2} \cdots s_{i_l}s_{j_1}s_{j_2}\cdots s_{j_m}, t)$, where $\{i_1,i_2,\ldots,i_l,j_1,j_2,\ldots,j_m\}\subseteq\{1,2,\ldots,k\}$.
\end{lem}
\begin{proof}
By assumption, deleting some columns from an OA$(N, s_1 s_2 \cdots s_k, t)$ we obtain an OA$(N,\\ s_{i_1} s_{i_2} \cdots s_{i_l}s_{j_1}s_{j_2}\cdots s_{j_m}, t)$ such that $N=s_{i_1}s_{i_2}\cdots s_{i_l}$ and every possible $l$-tuple in the columns $i_1,i_2,\ldots,i_l$ occurs exactly once in the OA as a row, where $\{i_1,i_2,\ldots,i_l,j_1,j_2,\ldots,j_m\}\subseteq\{1,2,\ldots,k\}$.
The conclusion immediately holds by Lemma \ref{3-20}.
\end{proof}

By Lemmas \ref{3-201} and \ref{3-202}, the new LOAs can be obtained
in \textbf{Table 1}, where the known OAs listed in the following
table are from  the monograph [24] or the website:
\begin{center}
http://neilsloane.com/oadir/index.html.
\end{center}
We also found some OAs satisfying the conditions mentioned in Lemma \ref{3-20} and the corresponding columns are marked in red, which are listed in \textbf{Appendix A}. Then by Lemmas \ref{3-20}-\ref{3-201}, the new LOAs can be also obtained and listed in \textbf{Table 2}.

\begin{center}
\textbf{Table 1 \ New LOAs }
\vskip 8pt
{\scriptsize
\begin{tabular}{lll}
\hline Known OAs &  New LOAs& Parameters \\
\hline
OA$(8,2^4 4^1,2)$&LOA$(8,2^k 4^1,2)$& $1\leq k\leq 4$\\
OA$(12,2^2 6^1,2)$&LOA$(12,2^k 6^1,2)$& $k=1,2$\\
OA$(16,2^8 8^1,2)$&LOA$(16,2^k 8^1,2)$& $1\leq k\leq 8$\\
OA$(18,3^6 6^1,2)$&LOA$(18,3^k 6^1,2)$& $1\leq k\leq 6$\\
OA$(20,2^2 10^1,2)$&LOA$(20,2^k 10^1,2)$& $k=1,2$\\
OA$(24, 2^{12} 12^1,2)$&LOA$(24, 2^{k} 12^1,2)$& $1\leq k\leq 12$\\
OA$(24,2^{11} 4^1 6^1,2)$&LOA$(24,2^{k} 4^1 6^1,2)$& $0\leq k\leq 11$\\
OA$(27,3^9 9^1,2)$&LOA$(27,3^k 9^1,2)$& $1\leq k\leq 9$\\
OA$(28,2^2 14^1,2)$&LOA$(28,2^k 14^1,2)$& $ k=1,2$\\
OA$(32,2^{16} 16^1,2)$&LOA$(32,2^{k} 16^1,2)$& $1\leq k\leq 16$\\
OA$(32,4^8 8^1,2)$&LOA$(32,4^k 8^1,2)$& $1\leq k\leq 8$\\
OA$(36,2^{10} 3^1 6^2 ,2)$&LOA$(36,2^{k_1} 3^{k_2} 6^2 ,2)$&$0\leq k_1\leq 10$, $k_2=0,1$\\
OA$(36,2^9 3^4 6^2,2)$&LOA$(36,2^{k_1} 3^{k_2} 6^2,2)$& $0\leq k_1\leq 9$, $0\leq k_2\leq4$\\
OA$(36,3^{12} 12^1,2)$&LOA$(36,3^{k} 12^1,2)$&$1\leq k\leq 12$\\
OA$(36,2^{13}6^2,2)$&LOA$(36,2^{k}6^2,2)$&$0\leq k\leq 13$\\
OA$(40,2^{19} 4^1 10^1,2)$ &LOA$(40,2^{k}4^1 10^1,2)$&$0\leq k\leq 19$\\
OA$(45,3^9 15^1,2)$&LOA$(45,3^k 15^1,2)$&$1\leq k\leq 9$\\
OA$(48,2^{31} 6^1 8^1,2)$&LOA$(48,2^{k} 6^1 8^1,2)$&$0\leq k\leq 31$\\
OA$(48,2^{24} 24^1,2)$&LOA$(48,2^{k} 24^1,2)$&$1\leq k\leq 24$\\
OA$(50,5^{10} 10^1,2)$&LOA$(50,5^{k} 10^1,2)$&$1\leq k\leq 10$ \\
OA$(54,3^{20} 6^1 9^1,2)$& LOA$(54,3^{k} 6^1 9^1,2)$&$0\leq k\leq 20$\\
OA$(56,2^{27} 4^1 14^1,2)$&LOA$(56,2^{k} 4^1 14^1,2)$&$0\leq k\leq 27$\\
OA$(60,2^{15}6^110^1,2)$&LOA$(60,2^{k}6^110^1,2)$&$0\leq k\leq 15$\\
OA$(63,3^{12} 21^1,2)$&LOA$(63,3^{k} 21^1,2)$&$1\leq k\leq 12$\\
OA$(64,2^{32} 32^1,2)$&LOA$(64,2^{k} 32^1,2)$&$1\leq k\leq 32$\\
OA$(64,4^{16} 16^1,2)$&LOA$(64,4^{k} 16^1,2)$&$1\leq k\leq 16$\\
OA$(64,4^7 8^6,2)$&LOA$(64,4^{k_1}8^{k_2},2)$&$0\leq k_1\leq 7$, $2\leq k_2\leq6$\\
OA$(72,2^{27}3^{11}6^112^1,2)$&LOA$(72,2^{k_1}3^{k_2}6^112^1,2)$&$0\leq k_1\leq 27$, $0\leq k_2\leq11$\\
OA$(80,2^{55} 8^1 10^1,2)$&LOA$(80,2^{k} 8^1 10^1,2)$&$0\leq k\leq 55$\\
OA$(80,2^{51} 4^3 20^1)$&LOA$(80,2^{k_1} 4^{k_2} 20^1)$&$0\leq k_1\leq 51$, $1\leq k_2\leq3$\\
OA$(80,2^{40} 40^1,2)$&LOA$(80,2^{k} 40^1,2)$&$1\leq k\leq 40$\\
OA$(81,3^{27} 27^1)$&LOA$(81,3^{k} 27^1)$&$1\leq k\leq 27$\\
OA$(84,2^{14}6^114^1,2)$&LOA$(84,2^{k}6^114^1,2)$&$0\leq k\leq 14$\\
OA$(88,2^{44} 44^1 ,2)$&LOA$(88,2^{k} 44^1 ,2)$&$1\leq k\leq 44$\\
OA$(90,3^{30} 30^1,2)$&LOA$(90,3^{k} 30^1,2)$&$1\leq k\leq 30$\\
OA$(90,3^{26} 6^1 15^1,2)$&LOA$(90,3^{k} 6^1 15^1,2)$&$0\leq k\leq 26$\\
OA$(96,2^{71} 6^1 16^1,2)$&LOA$(96,2^{k} 6^1 16^1,2)$& $0\leq k\leq 71$\\ OA$(132,2^26^122^1,2)$&LOA$(132,2^{k}6^122^1,2)$&$0\leq k\leq 2$\\
OA$(16,2^3 4^1,3)$&LOA$(16,2^k 4^1,3)$&$k=2,3$\\
OA$(24,2^3 6^1,3)$&LOA$(24,2^k 6^1,3)$&$k=2,3$\\
OA$(32,2^44^2,3)$&LOA$(32,2^k4^2,3)$&$1\leq k\leq 4$\\
OA$(48,2^4 6^1, 4)$&LOA$(48,2^k 6^1, 4)$&$k=3,4$\\
OA$(128,2^34^3,4)$&LOA$(128,2^k4^3,4)$&$k=1,2,3$\\
OA$(128,2^4 4^2, 5)$&LOA$(128,2^k 4^2, 5)$&$k=3,4$\\
\hline
\end{tabular}}
\end{center}

\newpage

\begin{center}
\textbf{Table 2 \ New LOAs with runs size $N\leq100$ }
\vskip 8pt
{%\footnotesize
\scriptsize
\begin{tabular}{ll|ll}
\hline New LOAs &  Parameters& New LOAs &  Parameters\\
\hline
LOA$(20,2^k5^1,2)$&$2\leq k\leq 8$& LOA$(24, 2^k 3^1 4^1,2)$&$1\leq k\leq 13$\\
LOA$(28,2^k 7^1,2)$&$2\leq k\leq 12$&LOA$(36,6^1 3^{k_1} 2^{k_2},2)$&$1\leq k_1\leq 12$, $1\leq k_2\leq2$\\
LOA$(36,2^{k}3^1 6^1,2)$&$1\leq k\leq 18$&LOA$(36,2^{k}9^1,2)$&$2\leq k\leq 13$\\
LOA$(36,2^{k_1} 3^{k_2}6^1,2)$&$1\leq k_1\leq 11$, $1\leq k_2\leq 2$&LOA$(36,6^1 3^{k_1} 2^{k_2},2)$&$1\leq k_1\leq 8$, $1\leq k_2\leq10$\\
LOA$(36,3^2 2^k,2)$&$2\leq k\leq 20$&LOA$(40,2^{k} 4^1 5^1,2)$&$1\leq k\leq 25$\\
LOA$(44,2^{k}11^1,2)$&$2\leq k\leq 16$&LOA$(48,2^{k} 3^1 8^1,2)$&$1\leq k\leq 33$\\
LOA$(52,2^{k} 13^1 ,2)$&$2\leq k\leq 17$&LOA$(56,2^{k} 4^1 7^1 ,2)$&$2\leq k\leq 37$\\
LOA$(64,2^{k_1} 4^{k_2} 8^1 ,2)$&$1\leq k_1\leq 5$, $1\leq k_2\leq17$&LOA$(72,2^{k} 3^1 4^1 6^1,2)$&$0\leq k\leq 51$\\
LOA$(72,2^{k} 4^1 9^1,2)$&$1\leq k\leq 49$&LOA$(72,2^{k_1} 4^{k_2} 6^2,2)$&$1\leq k_1\leq 46$, $0\leq k_2\leq1$\\
LOA$(72,2^{k_1} 3^{k_2} 4^1,2)$&$1\leq k_1\leq 44$, $2\leq k_2\leq12$&LOA$(72,2^{k_1} 3^{k_2} 4^{k_3} 6^2,2)$&$1\leq k_1\leq 42$, $0\leq k_2\leq4$, $0\leq k_3\leq 1$\\
LOA$(72,2^{k_1} 4^{k_2} 6^{k_3} ,2)$&$1\leq k_1\leq 41$,$0\leq k_2\leq 1$, $2\leq k_3\leq3$&LOA$(72,2^{k_1} 3^{k_2} 4^1 6^1,2)$&$0\leq k_1\leq 36$, $1\leq k_2\leq9$\\
LOA$(72,2^{k_1} 3^{k_2} 4^1 6^1,2)$&$0\leq k_1\leq 35$, $1\leq k_2\leq12$&LOA$(72,2^{k_1} 3^{k_2} 4^{k_3} 6^2,2)$&$1\leq k_1\leq 34$, $0\leq k_2\leq8$, $0\leq k_3\leq 1$\\
LOA$(72,2^{k_1} 3^{k_2} 6^{k_3} ,2)$&$1\leq k_1\leq 30$, $0\leq k_2\leq 1$, $2\leq k_3\leq4$&LOA$(80,2^{k} 5^1 8^1 ,2)$&$1\leq k\leq 61$\\
LOA$(96,2^{k} 3^1 16^1  ,2)$&$1\leq k\leq 73$&LOA$(96,2^{k_1} 4^{k_2} 6^1 8^1 ,2)$&$1\leq k_1\leq 43$, $0\leq k_2\leq12$\\
LOA$(96,2^{k_1} 3^1 4^{k_2} 8^1 ,2)$&$0\leq k_1\leq 39$, $1\leq k_2\leq14$
&LOA$(96,2^{k_1} 3^1 4^{k_2} ,2)$&$1\leq k_1\leq 19$, $2\leq k_2\leq23$\\
LOA$(96,2^{k_1} 4^{k_2} 12^1 ,2)$&$1\leq k_1\leq 18$, $1\leq k_2\leq22$
&LOA$(96,2^{k_1} 4^{k_2} 6^1 ,2)$&$2\leq k_1\leq 17$, $1\leq k_2\leq23$\\
LOA$(100,2^{k_1} 5^{k_2} ,2)$&$2\leq k_1\leq 40$, $2\leq k_2\leq4$&LOA$(40, 5^12^k, 3)$&$3\leq k\leq 6$\\
LOA$(48, 4^1 3^1 2^k, 3)$&$2\leq k\leq 4$&LOA$(48, 3^1 2^k, 3)$&$4\leq k\leq 9$\\
LOA$(54, 3^k2^1, 3)$&$3\leq k\leq 5$\\ \hline
\end{tabular}}
\end{center}

\subsection{Juxtaposition Construction }

The following is the juxtaposition construction of orthogonal arrays.

\begin{lem}[\cite{Chen}]\label{juxtaposition-OA}
If there exist an OA$(N_1, a_1a_2 \cdots a_k, t) $ and an OA$(N_2, b_1a_2 \cdots a_k, t) $ satisfying $\frac {N_1} {a_1}= \frac {N_2} {b_1} $, then there exists an OA $(N_1+N_2, (a_1+b_1)a_2\cdots a_k, t)$.
\end{lem}

In a similar way to the juxtaposition construction of orthogonal arrays, we have the following theorem.

\begin{thm}\label{N1+N2}
If there exist an LOA$(N_1, a_1a_2\cdots a_k, t) $ and an LOA$(N_2, b_1a_2\cdots a_k, t)$ satisfying $\frac{N_1}{a_1}=\frac{N_2}{b_1} $, then there exists an LOA$(N_1+N_2, (a_1+b_1)a_2\cdots a_k, t)$.
\end{thm}
\begin{proof}
We have $\frac{\prod_ {i=1}^{k}a_ i}{N_1}=\frac{b_1\prod_{i=2}^{k}a_ i}{N_2}$
since $\frac{N_1}{a_1}=\frac{N_2}{b_1}$ and denote the number as $w$. Let $\{A_1, A_2,  \cdots,\\ A_w\}$ and $\{B_1, B_2, \cdots, B_w\}$ be an LOA $(N_1, a_1a_2\cdots a_k, t)$
and an LOA $(N_2, b_1a_2\cdots a_k, t) $ over $X_1\times X_2\times\cdots\times X_k$ and
$Y_1\times X_2\times\cdots\times X_k$ respectively, where $X_1\cap Y_1=\emptyset$. Let
\begin{center}
$M_i=\left(\begin{matrix}
A_i\\
B_i\\
\end{matrix}\right)$, $1 \leq i \leq w$.
\end{center}
By Lemma \ref{juxtaposition-OA}, each $M_ {i} $, $1 \leq i \leq w$, is an
OA$(N_1+N_2, (a_1+b_1)a_ 2\cdots a_ k, t)$ over $(X_1\cup Y_1)\times X_2\times\cdots\times X_k$. Furthermore, it is easy to verify that each $k$-tuple over $(X_1\cup Y_1)\times X_2\times\cdots\times X_k$ occurs in exactly one $M_i$ as a row. It follows that $\{M_{i}: 1  \leq i  \leq w \}$ is an LOA$(N_1+N_2, (a_1+b_1)a_ 2\cdots a_ k, t)$ over $(X_1\cup Y_1)\times X_2\times\cdots\times X_k$.
\end{proof}

\begin{exam}
There exists an LOA$(56,2^67^1,3)$.
\end{exam}
\begin{proof}
There exist an LOA$(16,7,2,3)$ and an LOA$(40,2^6 5^1,3)$ by Theorem \ref{LOA-level2-str3} and \textbf{Table 2} respectively. Clearly, $\frac{16}{2}=\frac{40}{5}$. Then there exists an LOA$(56,2^67^1,3)$ by Theorem \ref{N1+N2}.
\end{proof}

By Theorem \ref{N1+N2}, the new LOAs can be obtained and listed in \textbf{Table 3},
where the LOAs in the first column are from \textbf{Table 1} and \textbf{Table 2} respectively.

\begin{center}
\textbf{Table 3 \ New LOAs }
\vskip 8pt
{\footnotesize
\begin{tabular}{lll}
\hline Known LOAs &  Condition & New LOAs\\
\hline
LOA$(20,5^12^8,2)$ and LOA$(28,7^12^8,2)$&$\frac{20}{5}=\frac{28}{7}$&LOA$(48,12^12^8,2)$\\
LOA$(24,4^13^12^{13},2)$ and LOA$(36,6^13^12^{13},2)$&$\frac{24}{4}=\frac{36}{6}$&LOA$(60,{10}^13^12^{13},2)$\\
LOA$(16,4^12^9,2)$ and LOA$(44,11^12^9,2)$&$\frac{16}{4}=\frac{44}{11}$&LOA$(60,15^12^9,2)$\\
LOA$(20,5^12^8,2)$ and LOA$(44,11^12^8,2)$&$\frac{20}{5}=\frac{44}{11}$&LOA$(64,16^12^8,2)$\\
LOA$(24,4^13^12^{13},2)$ and LOA$(40,5^14^12^{13},2)$&$\frac{24}{3}=\frac{40}{5}$&LOA$(64,8^14^12^{13},2)$\\
LOA$(16,4^12^9,2)$ and LOA$(52,13^12^9,2)$&$\frac{16}{4}=\frac{52}{13}$&LOA$(68,17^12^9,2)$\\
LOA$(28,7^12^{12},2)$ and LOA$(44,11^12^{12},2)$&$\frac{28}{7}=\frac{44}{11}$&LOA$(72,18^12^{12},2)$\\
LOA$(36,9^12^{13},2)$ and LOA$(40,10^12^{13},2)$&$\frac{36}{9}=\frac{40}{10}$&LOA$(76,19^{1}2^{13},2)$\\
LOA$(44,11^12^{13},2)$ and LOA$(36,9^12^{13},2)$&$\frac{44}{11}=\frac{36}{9}$&LOA$(80,20^12^{13},2)$\\
LOA$(36,6^13^12^{13},2)$ and LOA$(48,8^13^12^{13},2)$&$\frac{36}{6}=\frac{48}{8}$&LOA$(84,{14}^13^12^{18},2)$\\
LOA$(36,9^12^{13},2)$ and LOA$(48,12^12^{13},2)$&$\frac{36}{9}=\frac{48}{12}$&LOA$(84,{21}^12^{13},2)$\\
LOA$(28,7^12^{12},2)$ and LOA$(56,7^14^12^{11},2)$&$\frac{28}{2}=\frac{56}{4}$&LOA$(84,{7}^16^12^{11},2)$\\
LOA$(24,6^12^{13},2)$ and LOA$(68,17^12^{13},2)$&$\frac{24}{6}=\frac{68}{17}$&LOA$(92,{23}^12^{13},2)$\\
LOA$(44,11^12^{16},2)$ and LOA$(52,13^12^{16},2)$&$\frac{44}{11}=\frac{52}{13}$&LOA$(96,{24}^12^{16},2)$\\
LOA$(24,4^13^12^{13},2)$ and LOA$(56,7^14^12^{13},2)$&$\frac{24}{3}=\frac{56}{7}$&LOA$(80,10^14^12^{13},2)$\\
LOA$(24,4^13^12^{13},2)$ and LOA$(48,8^13^12^{13},2)$&$\frac{24}{4}=\frac{48}{8}$&LOA$(72,12^13^12^{13},2)$\\
LOA$(52,13^12^{13},2)$ and LOA$(36,9^12^{13},2)$&$\frac{52}{13}=\frac{36}{9}$&LOA$(88,{22}^12^{13},2)$\\
LOA$(40,10^12^{9},2)$ and LOA$(68,17^12^{9},2)$&$\frac{40}{10}=\frac{68}{17}$&LOA$(108,{27}^12^{9},2)$\\
LOA$(36,9^12^{13},2)$ and LOA$(80,20^12^{13},2)$&$\frac{36}{9}=\frac{80}{20}$&LOA$(116,{29}^12^{13},2)$\\
LOA$(40,5^14^12^{25},2)$ and LOA$(80,8^15^12^{25},2)$&$\frac{40}{4}=\frac{80}{8}$&LOA$(120,{12}^15^12^{25},2)$\\
LOA$(72,9^14^12^{37},2)$ and LOA$(56,7^14^12^{37},2)$&$\frac{72}{9}=\frac{56}{7}$&LOA$(128,16^14^12^{37},2)$\\
LOA$(64,16^12^{9},2)$ and LOA$(68,17^12^{9},2)$&$\frac{64}{16}=\frac{68}{17}$&LOA$(132,33^12^{9},2)$\\
LOA$(32,10,2,3)$ and LOA$(48,3^12^{9},3)$&$\frac{32}{2}=\frac{48}{3}$&LOA$(80,5^12^{9},3)$\\
LOA$(32,4^12^5,3)$ and LOA$(48,4^13^12^{4},3)$&$\frac{32}{2}=\frac{48}{3}$&LOA$(80,5^14^12^{4},3)$\\
\hline
\end{tabular}}
\end{center}

\subsection{Hadamard Construction}

In this subsection, we shall use Hadamard matrices to construct large sets of orthogonal arrays.

A \emph{Hadamard matrix of order $n$} is an $n\times n$ matrix $H$ in which
every entry is $\pm1$ such that $HH^T = nI_n$.

Let $S_n=(s_\textbf{x, y})$ be the $2^n\times 2^n$ matrix in which the rows and columns are indexed by $\mathbb{Z}_2^n$ (in lexicographic order) and
$s_\textbf{x, y}=(-1)^{\textbf{x}\cdot \textbf{y}}$ for all $\textbf{x}, \textbf{y}$ $\in \mathbb{Z}_2^n$, where $\cdot$ denotes the inner product. $S_n$ is called \emph{the Sylvester matrix of order $2^n$}.

\begin{lem}[\cite{Stinson}]\label{S_n}
$S_n$ is a Hadamard matrix of order $2^n$.
\end{lem}
\begin{lem}[\cite{Hedayat-OA,Handbook}]\label{OA=Hadamard}
The following are equivalent:\\
(i) a Hadamard matrix of order $n$ exists;\\
(ii) an OA$(n,n-1,2,2)$ exists;\\
(iii) an OA$(2n,n,2,3)$ exists.
\end{lem}

Suppose $H_1 = (h_{i,j})$ is a Hadamard matrix of order $n_1$
and $H_2$ is a Hadamard matrix of order $n_2$.
We define the \emph{Kronecker Product} $H_1\otimes H_2$ to be the matrix of order $n_1n_2$ obtained by replacing every entry $h_{i,j}$ of $H_1$ by the $n_2\times n_2$ matrix $h_{i,j}H_2$ (where $xH_2$ denotes the matrix obtained
from $H_2$ by multiplying every entry by $x$).

\begin{thm}
$S_n=\underbrace{S_1\otimes S_1\otimes\cdots\otimes S_1}_{n}$.
\end{thm}
\begin{proof}
The proof is by mathematical induction on $n$.

When $n=1$, there is nothing to do.

When $n=2$, it is clear that
\begin{center}
$S_2=\bordermatrix{&00&01&10&11\cr
00&1&1&1&1\cr
01&1&-1&1&-1\cr
10&1&1&-1&-1\cr
11&1&-1&-1&1\cr}=S_1\otimes S_1$,
\end{center}
the conclusion immediately holds.

Suppose that $n\leq k$ and that the theorem is true for $n$. Now we shall prove that the conclusion is also true for $n=k+1$, that is,
\begin{center}
$S_{k+1}=S_1\otimes S_k=\underbrace{S_1\otimes S_1\otimes\cdots\otimes S_1}_{k+1}$.
\end{center}
By assumption, let the rows and columns of $S_k$ be labeled with $a_1a_2\cdots a_k\in\mathbb{Z}_2^{k}$ in lexicographic order. It is clear that $0a_1a_2\cdots a_k$, $1a_1a_2\cdots a_k$ (where $a_1a_2\cdots a_k\in\mathbb{Z}_2^{k}$ in lexicographic order) are in lexicographic order of $\mathbb{Z}_2^{k+1}$. So by the definition of $S_{k+1}$, we have
\begin{center}
$S_{k+1}=\bordermatrix{&0a_1\cdots a_k&1b_1\cdots b_k\cr
0a_1\cdots a_k&S_k&S_k\cr
1b_1\cdots b_k&S_k&-S_k\cr}=S_1\otimes S_k=\underbrace{S_1\otimes S_1\otimes\ldots\otimes S_1}_{k+1}.$
\end{center}
This completes the proof.
\end{proof}

\begin{thm}\label{LOA$(2^n)$}
There exists an LOA$(2^n,k,2,2)$ for any integer $n\geq2$ and $n\leq k\leq 2^n-1$.
\end{thm}
\begin{proof}
By Lemma \ref{S_n}, $S_n$ is a Hadamard matrix of order $2^n$. We can get an OA$(2^n,2^n-1,2,2)$, denoted by $A$, by deleting the first column
(labeled with $00\cdots 0$) of $S_n$. By Lemma \ref{3-201} and Lemma \ref{3-202},
to prove the conclusion, it is sufficient to show that there exist $n$ columns in
$A$ such that each $n$-tuple of $\mathbb{Z}_2$ occurs exactly once in these columns of $A$ as a row.
Now we consider $n$ columns labeled with $00 \cdots 01$, $00 \cdots 01 0$, $\ldots$, $10 \cdots 00$ (each $n$-tuple contains exactly one 1) of $A$.
We choose any two distinct rows labeled with $b_1b_2\cdots b_n$ and $c_1c_2\cdots c_n$ respectively, where $b_i, c_i\in \mathbb{Z}_2, 1\leq i\leq n$.
Then we have two $n$-tuples $((-1)^{b_n}, (-1)^{b_{n-1}}, \ldots, (-1)^{b_1})$ and
 $((-1)^{c_n}, (-1)^{c_{n-1}}, \ldots, (-1)^{c_1})$
in two rows labeled with $b_1b_2\cdots b_n$ and $c_1c_2\cdots c_n$ and $n$ columns labeled with $00 \cdots 01$, $00 \cdots 01 0$, $\ldots$, $10 \cdots 00$.
It is clear that
\begin{center}
$((-1)^{b_n}, (-1)^{b_{n-1}}, \ldots, (-1)^{b_1})\neq((-1)^{c_n}, (-1)^{c_{n-1}}, \ldots, (-1)^{c_1})$
\end{center}
 since $b_1b_2\cdots b_n\neq c_1c_2\cdots c_n$.
So any two $n$-tuples are distinct in the $n$ columns labeled with $00 \cdots 01$, $00 \cdots 01 0$, $\ldots$, $10 \cdots 00$ of $A$. It follows that each $n$-tuples occurs exactly once in these columns of $A$ as a row.
This completes the proof.
\end{proof}
\begin{rmk}
Theorem \ref{LOA$(2^n)$} had been shown via finite fields method by Chen and Niu firstly, see Theorem 2.21 in \cite{chen23}. Here we provide a new
proof method of Theorem \ref{LOA$(2^n)$} by Hadamard matrix.
\end{rmk}

\begin{thm}\label{LOA-level2-str3}
There exists an LOA$(2^{n+1},k,2,3)$ for any integer $n\geq2$
and $n+1\leq k\leq 2^n$.
\end{thm}
\begin{proof}
Let
\begin{center}
$M=\left(\begin{matrix}
S_n \\
\overline{S}_n\\
\end{matrix}\right),$
\end{center}
where $\overline{S}_n$ is obtained simply by interchanging the two symbols $1$ and $-1$
in the $S_n$. Then $M$ is an OA$(2^{n+1},k,2,3)$ over $\{-1,1\}$ by Theorem 2.24 in \cite{Hedayat-OA}. In a similar way to the proof of Theorem \ref{LOA$(2^n)$}, it is easy to check that any two $(n+1)$-tuples are distinct in the $n$ columns labeled with $00\cdots0$, $00 \cdots 01$, $00 \cdots 01 0$, $\ldots$, $10 \cdots 00$ of $M$. It follows that each $(n+1)$-tuples occurs exactly once in these columns of $M$ as a row.
By Lemma \ref{3-202}, there exists an LOA$(2^{n+1},k,2,3)$ for $n\geq2$ and $n+1\leq k\leq 2^n$.
\end{proof}

\subsection{Finite Field Construction}

Let $\mathbb{F}_q$ be a finite field of order $q$.
The following construction can be found in \cite{chen23}.

\begin{con}[\cite{chen23}]\label{Suen}
Suppose that m and l are two positive integers with $m \leq l$. If there exists
 an $m \times l$ matrix $M$, having entries from $\mathbb{F}_q$, which satisfies the following two properties:

 1. Any $t$ columns of $M$ are linearly independent, where $2\leq t \leq l$.

 2. There exist $m$ columns of $M$ such that they are linearly independent. \\
 Then there exists an LOA$(q^m,k,q,t)$, where $t\leq m \leq k \leq l$.
\end{con}

In order to utilise Construction \ref{Suen} well, we need the following lemma.

\begin{lem}\label{bukeyue}
There exists at least one element $a\in \mathbb{F}_{q}$ such that $g(x,y)=-(x^2+axy+y^2)\neq0$ when $(x,y)\neq(0,0)$.
\end{lem}
\begin{proof}
Without loss of generality, let $x\neq 0$, then we have
\begin{center}
$g(x,y)=-(x^2+axy+y^2)=-x^2((\frac{y}{x})^2+a\frac{y}{x}+1)\neq0$.
\end{center}
Set $\frac{y}{x}=z$. To prove the conclusion, it suffices to prove that
there exists at least one element $a\in \mathbb{F}_{q}$ such that
\begin{equation}
z^2+az+1\neq0.
\end{equation}
Denote $\mathbb{F}_{q}^*=\mathbb{F}_{q}\setminus\{0\}$. When $z=0$, it is clear that $z^2+az+1=1\neq0$ for any $a\in \mathbb{F}_{q}$. When $z\neq0$, let $\mathcal{A}=\{-z^{-1}(z^2+1)\ |\ z\in\mathbb{F}_{q}^*\}$. Obviously, $|\mathcal{A}|\leq q-1$. So there exists at least one element $a\in\mathbb{F}_{q}$ such that
$a\not\in\mathcal{A}$. It follows that there exists at least one element $a\in\mathbb{F}_{q}$ such that $a\neq-z^{-1}(z^2+1)$, that is, $z^2+az+1\neq0$.
This completes the proof.
\end{proof}

\begin{thm}\label{LOAq4-3}
Let $q$ be a prime power with $q\geq3$, then there exists an LOA$(q^4,k,q,3)$, where $4\leq k\leq q^2+1$.
\end{thm}
\begin{proof}
By Lemma \ref{bukeyue}, let $g(x,y)=-(x^2+axy+y^2)$ be the quadratic polynomial over $\mathbb{F}_q$, where $a\in \mathbb{F}_q$ and $g(x,y)\neq0$ when $(x,y)\neq(0,0)$.
Denote $\mathbb{F}_q=\{\alpha_0=0,\alpha_1=1,\alpha_2,\ldots,\alpha_{q-1}\}$. Let
 $$B_0=(0\ \ 0\ \ 1\ \ 0)^T, \ \ B_{u,v}=(\alpha_{u}\ \ \alpha_{v}\ \ g(\alpha_{u},\alpha_{v}) \ \ 1)^T, \ \ 0\leq u, v\leq q-1,$$
where $\alpha_{u},\alpha_{v} \in \mathbb{F}_q$ and
the symbol `` $T$ " represents the transposition of a vector.

Now we construct a $4\times (q^2+1)$ matrix $M$ as follows.
 \begin{center}
$M=(B_0\ \ B_{0,0} \ \ B_{0,1} \ \ \cdots\ \ B_{0,q-1}\ \ B_{1,0} \ \ \cdots \ \ B_{1,q-1} \ \ \cdots\ \ B_{q-1,0}\ \ \cdots \ \ B_{q-1,q-1})$
\end{center}

\ \ \ \ \ \ \ \ \ $=(A_0\ \ A_1 \ \ \cdots \ \ A_{q^2}).$

\noindent
It is easy to compute that the determinant of matrix $\left(A_0\ \ A_1\ \ A_2\ \ A_{q+1}\right)$ is
$$|A_0\ \ A_1\ \ A_2\ \ A_{q+1}|=|B_0\ \ B_{0,0}\ \ B_{0,1}\ \ B_{1,0}|=\left |\begin{smallmatrix}
  0 & 0 & 0 &1\\
  0 & 0 & 1&0 \\
  1 & 0&-1&-1 \\
  0 & 1& 1&1 \\
\end{smallmatrix}\right|
=-1\neq0.$$
So there exist $4$ columns of $M$ such that they are linearly independent. Next we shall
distinguish the following $4$ cases to show that any three columns of $M$ are linearly independent.

\textbf{Case 1:} When $i=0$, $j=1$, $k\in\{2,\cdots,q^2\}$. We have
$$(A_{i}\ \ A_{j}\ \ A_{k})=\left(
        \begin{smallmatrix}
              0 & 0 & \alpha_{u_k} \\
              0 & 0& \alpha_{v_k}\\
              1& 0 &g(\alpha_{u_{k}},\alpha_{v_{k}})\\
              0 & 1 & 1  \\
        \end{smallmatrix}
        \right).$$
It is clear that $(\alpha_{u_{k}},\alpha_{v_{k}})\neq(0,0)$,
it follows that the rank of $(A_{i}\ \ A_{j}\ \ A_{k})$ is $3$.

\textbf{Case 2:}\ When $i=0$, $\{j,k\}\subseteq\{2,\cdots,q^2\}$. We have
$$(A_{i}\ \ A_{j}\ \ A_{k})=\left(
        \begin{smallmatrix}
              0 & \alpha_{u_j} & \alpha_{u_k} \\
              0 & \alpha_{v_j} & \alpha_{v_k}\\
              1& g(\alpha_{u_{j}},\alpha_{v_{j}}) &g(\alpha_{u_{k}},\alpha_{v_{k}})\\
              0 & 1 & 1  \\
        \end{smallmatrix}
        \right).$$
It is clear that $(\alpha_{u_j},\alpha_{v_j})\neq(\alpha_{u_k},\alpha_{v_k})$, then
the rank of $(A_{i}\ \ A_{j}\ \ A_{k})$ is $3$.

\textbf{Case 3:}\  When $i=1$, $\{j,k\}\subseteq\{2,\cdots,q^2\}$. We have
$$(A_{i}\ \ A_{j}\ \ A_{k})=\left(
        \begin{smallmatrix}
              0 & \alpha_{u_j} & \alpha_{u_k} \\
              0 & \alpha_{v_j} & \alpha_{v_k}\\
              0& g(\alpha_{u_{j}},\alpha_{v_{j}}) &g(\alpha_{u_{k}},\alpha_{v_{k}})\\
              1 & 1 & 1  \\
        \end{smallmatrix}
        \right).$$

If the determinant
$\left |\begin{smallmatrix}
0 & \alpha_{u_{j}} & \alpha_{u_{k}} \\
0& \alpha_{v_{j}} &\alpha_{v_{k}}\\
  1 & 1& 1 \\
\end{smallmatrix}\right|
\neq0$, then the rank of $(A_{i}\ \ A_{j}\ \ A_{k})$ is $3$.

If $\left |\begin{smallmatrix}
0 & \alpha_{u_{j}} & \alpha_{u_{k}} \\
0& \alpha_{v_{j}} &\alpha_{v_{k}}\\
  1 & 1& 1 \\
\end{smallmatrix}\right|
=0$, then the three column vectors $(0,0,1)^T, (\alpha_{u_{j}},\alpha_{v_{j}},1)^T$ and $(\alpha_{u_{k}},\alpha_{v_{k}},1)^T$ are linear correlation.
So there exist two integers $x$ and $y$ such that
$$(\alpha_{u_{k}},\alpha_{v_{k}},1)^T=x(0,0,1)^T+y(\alpha_{u_{j}},\alpha_{v_{j}},1)^T.$$
It follows that
\begin{equation}
\left\{
          \begin{array}{l}
              \alpha_{u_{k}}=y\alpha_{u_{j}},\\
\alpha_{v_{k}}=y\alpha_{v_{j}},\\
1=x+y.
          \end{array}
       \right.
\end{equation}
Obviously, we have $A_j=A_k$ if $x=0$ and $A_i=A_k$ if $y=0$, which contradict the fact
that $A_i, A_j$ and $A_k$ are pair distinct. Thus $x,y\not\in\{0,1\}$.

We shall prove that
\begin{center}
$\left |\begin{smallmatrix}
   0& \alpha_{u_{j}} & \alpha_{u_{k}} \\
  0 & g(\alpha_{u_{j}},\alpha_{v_{j}}) &g(\alpha_{u_{k}},\alpha_{v_{k}})\\
  1 & 1& 1 \\
\end{smallmatrix}\right|\neq0$
or
$\left |\begin{smallmatrix}
   0& \alpha_{v_{j}} & \alpha_{v_{k}} \\
  0 & g(\alpha_{u_{j}},\alpha_{v_{j}}) &g(\alpha_{u_{k}},\alpha_{v_{k}})\\
  1 & 1& 1 \\
\end{smallmatrix}\right|\neq0$
\end{center}
by reduction to absurdity.
Suppose that
\begin{equation}\left |\begin{smallmatrix}
   0& \alpha_{u_{j}} & \alpha_{u_{k}} \\
  0 & g(\alpha_{u_{j}},\alpha_{v_{j}}) &g(\alpha_{u_{k}},\alpha_{v_{k}})\\
  1 & 1& 1 \\\\
\end{smallmatrix}\right|
=\alpha_{u_j}g(\alpha_{u_{k}},\alpha_{v_{k}})-\alpha_{u_{k}}g(\alpha_{u_{j}},\alpha_{v_{j}})=0,
\end{equation}
\begin{equation}
\left |\begin{smallmatrix}
   0& \alpha_{v_{j}} & \alpha_{v_{k}} \\
  0 & g(\alpha_{u_{j}},\alpha_{v_{j}}) &g(\alpha_{u_{k}},\alpha_{v_{k}})\\
  1 & 1& 1 \\\\
\end{smallmatrix}\right|
=\alpha_{v_j}g(\alpha_{u_{k}},\alpha_{v_{k}})-\alpha_{v_{k}}g(\alpha_{u_{j}},\alpha_{v_{j}})=0.
\end{equation}
Substituting (2) into (3) and (4), we have $$y^2\alpha_{u_j}g(\alpha_{u_{j}},\alpha_{v_{j}})=y\alpha_{u_j}g(\alpha_{u_{j}},\alpha_{v_{j}}),$$
$$y^2\alpha_{v_j}g(\alpha_{u_{j}},\alpha_{v_{j}})=y\alpha_{v_j}g(\alpha_{u_{j}},\alpha_{v_{j}}).$$
Then
$$y(y-1)\alpha_{u_j}g(\alpha_{u_{j}},\alpha_{v_{j}})=0,$$
$$y(y-1)\alpha_{v_j}g(\alpha_{u_{j}},\alpha_{v_{j}})=0.$$
Because $(\alpha_{u_{j}},\alpha_{v_{j}})\neq(0,0)$, without loss of generality, let $\alpha_{u_{j}}\neq0$. We also have $y\neq0,1$ and $g(\alpha_{u_{j}},\alpha_{v_{j}})\neq0$.
Thus
\begin{center}
$y(y-1)\alpha_{u_j}g(\alpha_{u_{j}},\alpha_{v_{j}})\neq0$.
\end{center}
This leads to a contradiction. It follows that the
matrix $(A_{i}\ \ A_{j}\ \ A_{k})$ has rank $3$.

\textbf{Case 4:}\ When $\{i,j,k\}\subseteq\{2,3,\cdots,q^2\}$. We have
$$(A_{i}\ \ A_{j}\ \ A_{k})=\left(
        \begin{smallmatrix}
              \alpha_{u_i} & \alpha_{u_j} & \alpha_{u_k} \\
              \alpha_{v_i} & \alpha_{v_j} & \alpha_{v_k}\\
              g(\alpha_{u_{i}},\alpha_{v_{i}})& g(\alpha_{u_{j}},\alpha_{v_{j}}) &g(\alpha_{u_{k}},\alpha_{v_{k}})\\
              1 & 1 & 1  \\
        \end{smallmatrix}
        \right).$$

If
$\left |\begin{smallmatrix}
  \alpha_{u_{i}} & \alpha_{u_{j}} & \alpha_{u_{k}} \\
\alpha_{v_{i}} & \alpha_{v_{j}} &\alpha_{v_{k}}\\
  1 & 1& 1 \\
\end{smallmatrix}\right|
\neq0$, then the rank of $(A_{i}\ \ A_{j}\ \ A_{k})$ is $3$.

If $\left |\begin{smallmatrix}
  \alpha_{u_{i}} & \alpha_{u_{j}} & \alpha_{u_{k}} \\
\alpha_{v_{i}} & \alpha_{v_{j}} &\alpha_{v_{k}}\\
  1 & 1& 1 \\
\end{smallmatrix}\right|
=0$, we shall show that
\begin{center}
$\left |\begin{smallmatrix}
   \alpha_{u_{i}} & \alpha_{u_{j}} & \alpha_{u_{k}} \\
  g(\alpha_{u_{i}},\alpha_{v_{i}}) & g(\alpha_{u_{j}},\alpha_{v_{j}}) &g(\alpha_{u_{k}},\alpha_{v_{k}})\\
  1 & 1& 1 \\\\
\end{smallmatrix}\right|
\neq0$ or
$\left |\begin{smallmatrix}
   \alpha_{v_{i}} & \alpha_{v_{j}} & \alpha_{v_{k}} \\
  g(\alpha_{u_{i}},\alpha_{v_{i}}) & g(\alpha_{u_{j}},\alpha_{v_{j}}) &g(\alpha_{u_{k}},\alpha_{v_{k}})\\
  1 & 1& 1 \\\\
\end{smallmatrix}\right|
\neq0.$
\end{center}
Clearly, the $3$ column vectors
$(\alpha_{u_{i}},\alpha_{v_{i}},1)^T,(\alpha_{u_{j}},\alpha_{v_{j}},1)^T$ and $(\alpha_{u_{k}},\alpha_{v_{k}},1)^T$ are linear correlation.
So there exist two integers $x$ and $y$ such that  $$(\alpha_{u_{k}},\alpha_{v_{k}},1)^T=x(\alpha_{u_{i}},\alpha_{v_{i}},1)^T+y(\alpha_{u_{j}},\alpha_{v_{j}},1)^T$$
then
\begin{equation}
\left\{
          \begin{array}{l}
              \alpha_{u_{k}}=x\alpha_{u_{i}}+y\alpha_{u_{j}},\\
\alpha_{v_{k}}=x\alpha_{v_{i}}+y\alpha_{v_{j}},\\
1=x+y.
          \end{array}
       \right.
\end{equation}
In a similar way to above case, we have $x,y\not\in\{0, 1\}$.

Suppose that
$\left |\begin{smallmatrix}
\alpha_{u_{i}} & \alpha_{u_{j}} & \alpha_{u_{k}} \\
g(\alpha_{u_{i}},\alpha_{v_{i}}) & g(\alpha_{u_{j}},\alpha_{v_{j}}) &g(\alpha_{u_{k}},\alpha_{v_{k}})\\
1 & 1& 1 \\\\
\end{smallmatrix}\right|
=0.$
Then the following system of equations holds.
\begin{equation}
\left\{
\begin{array}{l}
\alpha_{u_{k}}=x'\alpha_{u_{i}}+y'\alpha_{u_{j}},\\
g(\alpha_{u_{k}},\alpha_{v_{k}})=x'g(\alpha_{u_{i}},\alpha_{v_{i}}) +y'g(\alpha_{u_{j}},\alpha_{v_{j}}),\\
1=x'+y'.
\end{array}
\right.
\end{equation}
So we have $(x\alpha_{u_{i}}+y\alpha_{u_{j}})-(x'\alpha_{u_{i}}+y'\alpha_{u_{j}})=0$ and
$(x+y)-(x'+y')=0$. It follows that $(y-y')(\alpha_{u_{j}}-\alpha_{u_{i}})=0$, thus
$y'=y$ or $\alpha_{u_{i}}=\alpha_{u_{j}}$.

When $y'=y$. Clearly, $x'=x$.
Substituting (5) into (6), we have $$(x^2-x)(g(\alpha_{u_{i}},\alpha_{v_{i}})+g(\alpha_{u_{j}},\alpha_{v_{j}})-2(\alpha_{u_{i}}
\alpha_{u_{j}}+\alpha_{v_{i}}\alpha_{v_{j}})-a(\alpha_{u_{i}}\alpha_{v_{j}}+\alpha_{u_{j}}
\alpha_{v_{i}}))=0.$$
So
$$g(\alpha_{u_{i}},\alpha_{v_{i}})+g(\alpha_{u_{j}},\alpha_{v_{j}})-2(\alpha_{u_{i}}
\alpha_{u_{j}}+\alpha_{v_{i}}\alpha_{v_{j}})-a(\alpha_{u_{i}}\alpha_{v_{j}}+\alpha_{u_{j}}
\alpha_{v_{i}})=0.$$
Substituting $g(\alpha_{u},\alpha_{v})=-(\alpha_{u}^2+a\alpha_{u}\alpha_{v}+\alpha_{v}^2)$ into the above equality, we have
$$(\alpha_{u_{i}}-\alpha_{u_{j}})^2+a(\alpha_{u_{i}}-\alpha_{u_{j}})(\alpha_{v_{i}}-
\alpha_{v_{j}})+(\alpha_{v_{i}}-\alpha_{v_{j}})^2=0,$$
that is,
$$g(\alpha_{u_{i}}-\alpha_{u_{j}},\alpha_{v_{i}}-\alpha_{v_{j}})=0.$$
It follows that $\alpha_{u_{i}}-\alpha_{u_{j}}=0$ and $\alpha_{v_{i}}-\alpha_{v_{j}}=0$,
which contradict the fact $A_i, A_j$ and $A_k$ are pair distinct.
Thus $\left |\begin{smallmatrix}
\alpha_{u_{i}} & \alpha_{u_{j}} & \alpha_{u_{k}} \\
g(\alpha_{u_{i}},\alpha_{v_{i}}) & g(\alpha_{u_{j}},\alpha_{v_{j}}) &g(\alpha_{u_{k}},\alpha_{v_{k}})\\
1 & 1& 1 \\
\end{smallmatrix}\right|
\neq0.$
We have the rank of the matrix $(A_{i}\ \ A_{j}\ \ A_{k})$ is $3$.

When $\alpha_{u_{i}}=\alpha_{u_{j}}$. Clearly, $\alpha_{u_{i}}=\alpha_{u_{j}}=\alpha_{u_{k}}$. By a simple computation, we have
\begin{center}
$\left |\begin{smallmatrix}
   \alpha_{v_{i}} & \alpha_{v_{j}} & \alpha_{v_{k}} \\
  g(\alpha_{u_{i}},\alpha_{v_{i}}) & g(\alpha_{u_{j}},\alpha_{v_{j}}) &g(\alpha_{u_{k}},\alpha_{v_{k}})\\
  1 & 1& 1 \\
\end{smallmatrix}\right|
=(\alpha_{v_{j}}-\alpha_{v_{i}})(\alpha_{v_{k}}-\alpha_{v_{i}})(\alpha_{v_{k}}-\alpha_{v_{j}})\neq0.$
\end{center}
It follows that the rank of the matrix $(A_{i}\ \ A_{j}\ \ A_{k})$ is $3$.

By Construction \ref{Suen}, there exists an LOA$(q^4,k,q,3)$ for any prime power $q\geq 3$ and $4\leq k\leq q^2+1$. The proof is completed.
\end{proof}

\subsection{Difference Matrix Construction }

In this subsection, new construction methods for large sets of orthogonal arrays
are given by difference matrix.

Let $G$ be an abelian group of order $v$. A \emph{difference matrix}, or a $(v, k, 1)$-DM is a $v \times k$ array $D = (d_{i,j})$$(1 \leq i \leq v, 1 \leq j \leq k)$
with entries from $G$, such that for any two distinct columns $l$ and $h$ of $D(1 \leq l < h \leq k)$, the difference list $\bigtriangleup_{l,h} =
\{d_{1,h} - d_{1,l}, d_{2,h} - d_{2,l}, \ldots , d_{v,h} - d_{v,l}\}$ contains every element of $G$ exactly once.

\begin{con}\label{chai1}
If there exists a $(v,4,1)$-DM, then there exists an LOA$(v^{3},k,v,2)$, where $3\leq k\leq13$.
\end{con}
\begin{proof}
Let $D=(d_{i,j})_{v\times 4}$ be a $(v,4,1)$-DM over an abelian group $G$ of order $v$. For each row  $(d_{i,1},d_{i,2},d_{i,3},d_{i,4})$ of $D$, we construct the following rows:

\begin{center}
$C(i,u,e)=(d_{i,1}+u,d_{i,2}+u,d_{i,3}+u,d_{i,4}+u,d_{i,1}+u+e,
d_{i,2}+u+e,d_{i,3}+u+e,d_{i,4}+u+e,d_{i,1}-d_{i,2},d_{i,1}-d_{i,2}+e,d_{i,1}-d_{i,3}+e,
d_{i,1}-d_{i,4}+e,e)$,
\end{center}
where $u, e\in G$, $i\in\{1,2,\cdots,v\}$,
then the matrix is denoted by $M$. By Lemma $2.4$ in \cite{chen19}, $M$ is an OA$(v^{3},13,v,2)$ and the $1$-st, $2$-nd and $7$-th columns of $M$ constitute an OA$(v^{3},3,v,3)$. It follows that there exists an LOA$(v^{3},k,v,2)$  for $3\leq k\leq 13$ by Lemma \ref{3-202}.
\end{proof}

\begin{con}\label{chai2}
If there exists a $(v,4,1)$-DM, then there exists an LOA$(v^{4},k,v,2)$, where $4\leq k\leq29$.
\end{con}
\begin{proof}
Let $D=(d_{i,j})_{v\times 4}$ be a $(v,4,1)$-DM over an abelian group $G$ of order $v$. For each row  $(d_{i,1},d_{i,2},d_{i,3},d_{i,4})$ of $D$, we construct the following rows:

\begin{center}
$C(i,u,e,w)=(d_{i,1}+u,d_{i,2}+u,d_{i,3}+u+e,d_{i,4}+u+e+w,w,d_{i,3}+u+w,d_{i,4}+u+w,
d_{i,1}+u+e+w,d_{i,3}+u,d_{i,4}+u,d_{i,1}+u+e,d_{i,2}+u+e,
d_{i,4}+u+e,d_{i,1}+u+w,d_{i,2}+e+w,d_{i,3}+e+w,d_{i,4}+e+w,
d_{i,1}+u+w,d_{i,2}+u+w,d_{i,2}+u+e+w,d_{i,3}+u+e+w,
d_{i,1}-d_{i,2},d_{i,1}-d_{i,2}+e,d_{i,1}-d_{i,3}+e,d_{i,1}-d_{i,4}+e,
d_{i,1}-d_{i,2}+w,d_{i,1}-d_{i,3}+w,d_{i,1}-d_{i,4}+w,e)$,
\end{center}
where $u, e,w\in G$, $i\in\{1,2,\cdots,v\}$,
then the matrix is denoted by $M$. By Lemma $2.5$ in \cite{chen19}, $M$ is an OA$(v^{4},29,v,2)$. It is easy to verify that the first
four columns of $M$ constitute an OA$(v^{4},4,v,4)$. It follows that there exists an  LOA$(v^{4},k,v,2)$ for $4\leq k\leq 29$ by Lemma\ref{3-202}.
\end{proof}

We have the following result on $(v, 4, 1)$-DM.

\begin{lem}[\cite{g41}]\label{v41DM}
If $v\geq4$ and $v\not\equiv2 \pmod 4$, then there exists a $(v,4,1)$-DM.
\end{lem}
The following theorem is immediately obtained by combining Construction \ref{chai1}, Construction \ref{chai2} and Lemma \ref{v41DM}.

\begin{thm}\label{chai11}
If $v\geq4$ and $v\not\equiv2 \pmod 4$, there exist an LOA$(v^{3},k,v,2)$ for $3\leq k\leq13$ and an LOA$(v^{4},k,v,2)$ for $4\leq k\leq29$.
\end{thm}

\section{New results on orthogonal arrays}

In this section, we shall use the following constructions to obtain some new results on orthogonal arrays of strength $t\geq 4$. Let \emph{lcm}$\{s_1,s_2,\ldots,s_k\}$ and
\emph{gcd}$\{s_1,s_2,\ldots,s_k\}$
be the least common multiple and the greatest common
factor of the positive integers $s_1,s_2,\ldots,s_k$ respectively.

\begin{con}[\cite{chen23}]\label{3-1}
If an LOA$(N_1,m,p,t_1)$ and an LOA$(N_2,n,q,t_2)$ exist, then there exists an OA$(hN_1N_2, p^{m}q^{n}, t)$, where $h=lcm\{\frac{p^m}{N_1},\frac{q^n}{N_2}\}$ and $t=t_1+t_2+1$.
\end{con}

\begin{con}[\cite{chen23}]\label{3-2}
If an LOA$(N_1, {p_1}^{k_1} {p_2}^{k_2}\cdots {p_m}^{k_m},t_1)$ and an LOA$(N_2,{q_1}^{l_1}{q_2}^{l_2}\cdots {q_n}^{l_n},\\t_2)$ exist, then there exists an OA$(hN_1N_2, {p_1}^{k_1}\cdots {p_m}^{k_m}{q_1}^{l_1}\cdots {q_n}^{l_n},t)$,
where $h=lcm\{\frac{\prod_{i=1}^{m}{p_i}^{k_i}}{N_1},\\\frac{\prod_{j=1}^{n}{q_j}^{l_j}}{N_2}\}$ and $t=t_1+t_2+1$.
\end{con}

To obtain new OAs by Constructions \ref{3-1}-\ref{3-2},
we need more results on LOAs in the following.

\begin{lem}[\cite{chen23}]\label{M-1}
Let $s_1,s_2,\ldots,s_k$ be positive integers,
there exists an LOA$(N,s_1 s_2 \cdots s_k,1)$,
where $N=lcm\{s_1, s_2,\ldots, s_k\}$.
\end{lem}

\begin{lem}[\cite{chen23}]\label{02}
There exists an LOA$(q^{n},k, q, 2)$ for any prime power $q$,
where $2\leq n\leq k\leq \frac{q^{n}-1}{q-1}$.
\end{lem}

\begin{lem}[\cite{Etzion}]\label{Etzion}
If there exists an OA$(v^t, k, v, t)$, then there exists an LOA$(v^t, k, v, t)$.
\end{lem}

\begin{lem}\label{03}
(1) If $q=2^m$, $m\geq 1$, then there exists an LOA$(q^3,k,q,3)$, where $3\leq k\leq q+2$.

(2) If $q$ is a prime power and $q+1\geq t\geq 0$,
then there exists an LOA$(q^t, k, q,t)$, where $t\leq k\leq q+1$.
\end{lem}
\begin{proof}
(1) For $q=2^m$, $m\geq 1$, it is clear that there exists an OA$(q^3,k,q,3)$ for $3\leq k\leq q+2$ by deleting $q+2-k$ columns from an OA$(q^3,q+2,q,3)$ coming from \cite{Bush}. It follows that there exists an LOA$(q^3,k,q,3)$ by Lemma \ref{Etzion}.

(2) The proof is the same as that of (1).
\end{proof}

\begin{lem}[\cite{chen23}]\label{5,6}
(1) There exists an LOA$(s^2,4,s,2)$ for $s \not\in \{2,6\}$.

(2) There exists an LOA$(s^2,5,s,2)$ for $s \not\in \{2,3,6,10\}$.

(3) There exists an LOA$(s^2,6,s,2)$ for $s \not\in \{2,3,4,6,10,22\}$.

(4) There exists an LOA$(v^3,5,v,3)$ for $v\geq4$ and $v\not\equiv2 \pmod 4$.

(5) There exists an LOA$(v^3,6,v,3)$ for $v$ satisfying gcd$\{v,4\}\neq2$ and gcd$\{v,18\}\neq3$.

(6) There exists an LOA$((mv)^3,6,mv,3)$ for $m\in\{5,7,11\}$ and
any odd positive integer $v$.

(7) There exists an LOA$(s^t,t+1,s,t)$ for any positive integers $s$ and $t$.
\end{lem}

By using Constructions \ref{3-2}, we can obtain some new OA with small runs listed in the following table, where the LOAs in the first column of the following table are from Lemma \ref{M-1}, \textbf{Table 1}, \textbf{Table 2} and Example 3 in \cite{chen23}.

%\newpage
\begin{center}
\textbf{Table 4 \ New OAs with small runs}
\vskip 8pt
{\footnotesize
\begin{tabular}{lll}
\hline Known LOA &  Parameters in Constructions \ref{3-2} & New OAs\\
\hline
LOA$(2,3,2,1)$ and LOA$(44,11^12^4,2)$&$N_1=2, N_2=44,h=4$&OA$(352,11^12^7,4)$\\
LOA$(24,6^14^12^2,2)$ and LOA$(4,2,4,1)$&$N_1=24, N_2=4,h=4$&OA$(384,6^14^32^2,4)$\\
LOA$(2,3,2,1)$ and LOA$(52,13^12^4,2)$&$N_1=2, N_2=52,h=4$&OA$(416,13^12^7,4)$\\
LOA$(2,3,2,1)$and LOA$(68,17^12^4,2)$&$N_1=2, N_2=68,h=4$&OA$(544,17^12^7,4)$\\
LOA$(2,3,2,1)$and LOA$(76,19^12^4,2)$&$N_1=2, N_2=76,h=4$&OA$(608,19^12^7,4)$\\

LOA$(2,3,2,1)$and LOA$(84,14^16^12^2,2)$&$N_1=2, N_2=84,h=4$&OA$(672,14^16^12^5,4)$\\
LOA$(4,2,4,1)$ and LOA$(44,11^12^4,2)$&$N_1=4, N_2=44,h=4$&OA$(704,11^12^44^2,4)$\\

LOA$(2,3,2,1)$and LOA$(116,29^12^4,2)$&$N_1=2, N_2=116,h=4$&OA$(928,29^12^7,4)$\\
LOA$(2,3,2,1)$ and LOA$(120,12^110^12^2)$&$N_1=2, N_2=120,h=4$&OA$(960,12^110^12^5,4)$\\
LOA$(2,3,2,1)$and LOA$(132,12^16^12^2,2)$&$N_1=2, N_2=132,h=4$&OA$(1056,22^16^12^5,4)$\\
LOA$(2,3,2,1)$and LOA$(144,24^16^12^2,2)$&$N_1=2, N_2=144,h=4$&OA$(1152,24^16^12^5,4)$\\

LOA$(4,3,2,2)$ and LOA$(36,6^13^12^2,2)$&$N_1=4, N_2=36,h=2$&OA$(288,6^13^12^5,5)$\\
LOA$(2,3,2,1)$ and LOA$(48,4^13^12^4,3)$&$N_1=2, N_2=48,h=4$&OA$(384,3^14^12^7,5)$\\
LOA$(2,3,2,1)$ and LOA$(80,5^14^12^4,3)$&$N_1=2, N_2=80,h=4$&OA$(640,5^14^12^{7},5)$\\
LOA$(2,4,2,1)$ and LOA$(40,5^12^6,3)$&$N_1=2, N_2=40,h=8$&OA$(640,5^12^{10},5)$\\
LOA$(2,4,2,1)$ and LOA$(56,2^67^1,3)$&$N_1=2, N_2=56,h=8$&OA$(896,7^12^{10},5)$\\
LOA$(12,3^12^4,2)$and LOA$(20,5^12^4,2)$&$N_1=12, N_2=20,h=4$&OA$(960,5^13^12^8,5)$\\
LOA$(12,3^12^4,2)$and LOA$(24,4^13^12^3,2)$&$N_1=12, N_2=24,h=4$&OA$(1153,4^13^22^7,5)$\\
LOA$(12,3^12^4,2)$and LOA$(28,7^12^4,2)$&$N_1=12, N_2=28,h=4$&OA$(1344,7^13^12^8,5)$\\
LOA$(12,3^12^4,2)$and LOA$(16,6,2,3)$&$N_1=12, N_2=16,h=4$&OA$(768,3^12^{10},6)$\\
LOA$(8,4^12^4,2)$and LOA$(40,5^12^6,3)$&$N_1=8, N_2=40,h=8$&OA$(2560,5^14^12^{10},6)$\\
\hline
\end{tabular}}
\end{center}

\begin{thm}\label{v1+v3-2}
If $v\geq4$ and $v\not\equiv2 \pmod 4$, then there exist an OA$(v^{k+1},2k-2,v,4)$ for $5\leq k\leq 13$ and an OA$(v^{k+2},2k-1,v,4)$ for $14\leq k\leq 28$.
\end{thm}
\begin{proof}
For any positive integers $v$ and $k$, there exists an LOA$(v,k-2,v,1)$
by Lemma \ref{M-1}. For any integer $v\geq4$, $v\not\equiv2 \pmod 4$
there exist an LOA$(v^3,k,v,2)$ for $5\leq k\leq 13$ and an LOA$(v^4,k+1,v,2)$
for $14\leq k\leq 28$ by Theorem \ref{chai11}.
Then by applying Construction \ref{3-2}, there exist an OA$(v^{k+1},2k-2,v,4)$ for $5\leq k\leq 13$ and an OA$(v^{k+2},2k-1,v,4)$ for $14\leq k\leq 28$.
\end{proof}

\begin{thm}\label{doublev3-2}
If $v\geq4$ and $v\not\equiv2 \pmod 4$, then there exist an OA$(v^{k+3},2k,v,5)$ for $5\leq k\leq 13$ and an OA$(v^{k+4},2k,v,5)$ for $14\leq k\leq 29$.
\end{thm}
\begin{proof}
For $v\geq4$, $v\not\equiv2 \pmod 4$ and $5\leq k\leq 13$, there exist two LOA$(v^3,k,v,2)$'s by Theorem \ref{chai11}. By Construction \ref{3-2}, there exists an OA$(v^{k+3},2k,v,5)$.

In a similar to above, there exists an OA$(v^{k+4},2k,v,5)$ for $14\leq k\leq 29$ by Construction \ref{3-2} with two LOA$(v^4,k,v,2)$'s from Theorem \ref{chai11}.
\end{proof}

\begin{thm}\label{v1+q4-3}
Suppose that $v$ and $k_1$ are positive integers, $q$ is a prime power with $q\geq3$ and integer $k_2$ satisfying $4\leq k_2\leq{q^2+1}$, then there exists an OA$(v{q^4}{h},v^{k_1-2}{q^{k_2}},5)$, where $h=lcm\{v^{k_1-3},q^{k_2-4}\}$.
\end{thm}
\begin{proof}
For any positive integers $v$ and $k_1$, there exists an LOA$(v,{k_1}-2,v,1)$
by Lemma \ref{M-1}. By Theorem \ref{LOAq4-3},
there exists an LOA$(q^4,k_2,q,3)$ for any prime power $q\geq3$ and $4\leq k_2\leq{q^2+1}$.
By Construction \ref{3-2}, there exists an OA$(v{q^4}{h},v^{k_1-2}{q^{k_2}},5)$, where $h=lcm\{v^{k_1-3},q^{k_2-4}\}$.
\end{proof}

\begin{thm}\label{v12n-4}
Suppose that $v$ and $k_1$ are positive integers, and $n\geq2$, then there exist an OA$({2^n}v{h_1},v^{k_1-2}{2^{k_2}},4)$ for $n\leq k_2\leq{2^n-1}$ and
an OA$({2^{n+1}}v{h_2},v^{k_1-2}{2^{k_2'}},5)$ for $n+1\leq k_2'\leq{2^n}$, where $h_1=lcm\{v^{k_1-3},2^{k_2-n}\}$ and $h_2=lcm\{v^{k_1-3},2^{k_2'-n-1}\}$.
\end{thm}
\begin{proof}
For any positive integers $v$ and $k_1$, there exists an LOA$(v,{k_1}-2,v,1)$
by Lemma \ref{M-1}. For any integer $n\geq2$, $n\leq k_2\leq{2^n-1}$ and $n+1\leq k_2'\leq{2^n}$,
there exist an LOA$(2^n,k_2,2,2)$ and an LOA$(2^{n+1},k_2',2,2)$ by Theorem \ref{LOA$(2^n)$} and Theorem \ref{LOA-level2-str3} respectively.
Then by applying Construction \ref{3-2}, there exist an OA$({2^n}v{h_1},v^{k_1-2}{2^{k_2}},4)$ for $n\leq k_2\leq{2^n-1}$ and
an OA$({2^{n+1}}v{h_2},v^{k_1-2}{2^{k_2'}},5)$ for $n+1\leq k_2'\leq{2^n}$, where $h_1=lcm\{v^{k_1-3},2^{k_2-n}\}$ and $h_2=lcm\{v^{k_1-3},2^{k_2-n-1}\}$.
\end{proof}

\begin{thm}\label{qn2n-com}
Suppose that $q$ is a prime power and $n\geq2$. For $2\leq m\leq {k_1}\leq{\frac{q^m-1}{q-1}}$, $n\leq k_2\leq{2^n-1}$ and $n+1\leq k_2'\leq{2^n}$,  there exist an OA$({2^n}{q^m}{h},{q^{k_1}}{2^{k_2}},5)$ and an OA$({2^n}{q^m}{h'},{q^{k_1}}{2^{k_2'}},6)$, where $h=lcm\{q^{k_1-m},2^{k_2-n}\}$ and $h'=lcm\{q^{k_1-m},2^{k_2'-n-1}\}$.
\end{thm}
\begin{proof}
For $2\leq m\leq {k_1}\leq{\frac{q^m-1}{q-1}}$, there exists an LOA$(q^{n},k_1, q, 2)$
by Lemma \ref{02}. For $n\geq2$, $n\leq k_2\leq{2^n-1}$ and $n+1\leq k_2'\leq{2^n}$,
there exist an LOA$(2^n,k_2,2,2)$ and an LOA$(2^{n+1},k_2',2,3)$ by Theorem \ref{LOA$(2^n)$} and Theorem \ref{LOA-level2-str3} respectively.
By applying Construction \ref{3-1}, there exist an OA$({q^m}{2^n}{h},{q^{k_1}}{2^{k_2}},5)$  and an OA$({2^n}{q^m}{h'},{q^{k_1}}{2^{k_2'}},6)$, where $h=lcm\{q^{k_1-m},2^{k_2-n}\}$ and $h'=lcm\{q^{k_1-m},2^{k_2'-n-1}\}$.
\end{proof}

\begin{thm}\label{qn2v32=5}
Suppose that $q$ is a prime power, $v\geq4$ and $v\not\equiv2 \pmod 4$, and $2\leq m\leq {k_1}\leq{\frac{q^m-1}{q-1}}$. There exist an OA$(q^m{v^3}{h},q^{k_1}{v^{k_2}},5)$ for $4\leq k_2\leq 13$ and OA$(q^m{v^4}{h'},q^{k_1}{v^{k_2}},5)$ for
$4\leq k_2\leq 29$, where $h=lcm\{q^{k_1-m},v^{k_2-3}\}$ and $h'=lcm\{q^{k_1-m},v^{k_2-4}\}$.
\end{thm}
\begin{proof}
For $2\leq m\leq {k_1}\leq{\frac{q^m-1}{q-1}}$, there exists an LOA$(q^{n},k_1, q, 2)$
by Lemma \ref{02}. For any integer $v\geq4$ and $v\not\equiv2 \pmod 4$,
there exist an LOA$(v^3,k_2,v,2)$ for $4\leq k_2\leq 13$ and an LOA$(v^4,k_2,v,2)$
for $4\leq k_2\leq 29$ by Theorems \ref{chai11}.
By applying Construction \ref{3-1}, there exist an OA$(q^m{v^3}{h},q^{k_1}{v^{k_2}},5)$ for $4\leq k_2\leq 13$ and OA$(q^m{v^4}{h'},q^{k_1}{v^{k_2}},5)$ for
$4\leq k_2\leq 29$.
\end{proof}

\begin{thm}\label{qn2q43=6}
Suppose that $n$ is an integer, $p$ and $q$ are prime power with $p,q\geq3$, there exists an OA$({p^4}{q^n}{h},p^{k_1}q^{k_2},6)$ for $4\leq k_1\leq{p^2+1}$ and $2\leq n\leq {k_2}\leq \frac{q^n-1}{q-1}$, where $h=lcm\{p^{k_1-4},q^{k_2-n}\}$.
\end{thm}
\begin{proof}
For any prime power $p\geq3$ and $4\leq k_1\leq{p^2+1}$,
there exists an LOA$(p^4,k_1,p,3)$ by Theorem \ref{LOAq4-3}, and
for any prime power $q\geq3$ and $2\leq n\leq {k_2}\leq {\frac{q^n-1}{q-1}}$, there exists an LOA$(q^n,{k_2},q,2)$ by Lemma \ref{02}.
By Construction \ref{3-2}, there exists an OA$({p^4}{q^n}{h},p^{k_1}q^{k_2},6)$, where $h=lcm\{p^{k_1-4},q^{k_2-n}\}$.
\end{proof}

\begin{thm}\label{qn3q43=7}
Suppose that $n$ is a positive integer, $p$ and $q$ are prime power with $p,q\geq3$, there exists an OA$({p^4}{q^3}{h},p^{k_1}q^{k_2},7)$ for $4\leq k_1\leq{p^2+1}$ and $3\leq {k_2}\leq q+1$, where $h=lcm\{p^{k_1-4},q^{k_2-3}\}$.
\end{thm}
\begin{proof}
For any prime power $p\geq3$ and $4\leq k_1\leq{p^2+1}$,
there exists an LOA$(p^4,k_1,p,3)$ by Theorem \ref{LOAq4-3}, and
for any prime power $q\geq3$ and $2\leq n\leq {k_2}\leq q+1$, there exists an LOA$(q^3,{k_2},q,3)$ by Theorem \ref{03}.
By Construction \ref{3-2}, there exists an OA$({p^4}{q^3}{h},p^{k_1}q^{k_2},7)$, where $h=lcm\{p^{k_1-4},q^{k_2-3}\}$.
\end{proof}

\begin{thm}\label{q3323=7}
If $q$ is a power of $2$, $n\geq 2$, $3\leq k_1\leq q+2$, $n\leq k_2\leq 2^n-1$ and $n+1\leq k_2'\leq 2^n$, then there exist an OA$(2^{n}{q^3}h,{q^{k_1}}2^{k_2},6)$ and an OA$(2^{n+1}{q^3}h',{q^{k_1}}2^{k_2'},7)$, where $h=lcm\{q^{k_1-3},2^{k_2-n}\}$ and $h'=lcm\{q^{k_1-3},2^{k_2'-n-1}\}$.
\end{thm}
\begin{proof}
By assumption  there exist an LOA$(q^3,k_1, q, 3)$ and an LOA$(2^{n},k_2,2,2)$
by Lemma \ref{03} and Theorem \ref{LOA$(2^n)$} respectively.
So does an OA$(2^{n}{q^3}h,{q^{k_1}}2^{k_2},6)$ by Construction \ref{3-2}.

By assumption  there exist an LOA$(q^3,k_1, q, 3)$ and an LOA$(2^{n+1},k_2',2,3)$
by Lemma \ref{03} and Theorem \ref{LOA-level2-str3} respectively.
So does an OA$(2^{n+1}{q^3}h,{q^{k_1}}2^{k_2'},7)$ by Construction \ref{3-2}.
\end{proof}

\begin{thm}\label{t-1q43=t+3}
Suppose that $s$ and $t$ are integers with $s,t\geq 2$, $q\geq 3$ is a prime power
and integer $k$ satisfying $4\leq k\leq q^2+1$, then there exists an OA$({q^4}{s^{t-1}}h,{q^{k}}{s^t},t+3)$, where $h=lcm\{s,q^{k-4}\}$.
\end{thm}
\begin{proof}
By assumption there exist an LOA$(s^{t-1},t, s, t-1)$ and an LOA$(q^{4},k,q,3)$
by Lemma \ref{5,6} and Theorem \ref{LOAq4-3} respectively. By Construction \ref{3-2} there exists an OA$({q^4}{s^{t-1}}h,{q^{k}}{s^t},t+3)$.
\end{proof}

\begin{thm}\label{qt2n2-3}
If $q$ is a prime power, $n\geq 2$, $1\leq t\leq k_1\leq q+1$, $n\leq k_2\leq 2^n-1$ and $n+1\leq k_2'\leq 2^n$, then there exist an OA$(2^{n}{q^t}h,{q^{k_1}}2^{k_2},t+3)$ and an OA$(2^{n+1}{q^t}h',{q^{k_1}}2^{k_2'},t+4)$, where $h=lcm\{q^{k_1-t},2^{k_2-n}\}$ and $h'=lcm\{q^{k_1-t},2^{k_2'-n-1}\}$.
\end{thm}
\begin{proof}
By assumption there exist an LOA$(q^t,k_1, q, t)$ and an LOA$(2^n,k_2,2,2)$
by Lemma \ref{03} and Theorem \ref{LOA$(2^n)$} respectively.
By Construction \ref{3-2}, there exists an OA$(2^{n}{q^t}h,{q^{k_1}}2^{k_2},t+3)$.

In a similar way, there exist an LOA$(q^t,k_1, q, t)$ and an LOA$(2^{n+1},k_2',2,3)$
by Lemma \ref{03} and Theorem \ref{LOA-level2-str3} respectively.
By Construction \ref{3-2}, there exists an OA$(2^{n+1}{q^t}h',{q^{k_1}}2^{k_2'},t+4)$.
\end{proof}

\begin{thm}\label{tt-1n2-3}
Suppose that $s$, $t$ and $n$ are integers with $s,t,n\geq 2$,  $n\leq k_1\leq 2^n-1$ and $n+1\leq k_1'\leq 2^n$, then there exist an OA$(2^{n}{s^{t-1}}h,2^{k_1}s^t,t+2)$ and an OA$(2^{n+1}{s^{t-1}}h',2^{k_1'}s^t,t+3)$, where $h=lcm\{s,2^{k_1-n}\}$ and $h'=lcm\{s,2^{k_1'-n-1}\}$.
\end{thm}
\begin{proof}
By assumption there exist an LOA$(s^{t-1},t, s, t-1)$ and an LOA$(2^n,k_1,2,2)$
by Lemma \ref{5,6} and Theorem \ref{LOA$(2^n)$} respectively.
By Construction \ref{3-2}, there exists an OA$(2^{n}{s^{t-1}}h,2^{k_1}s^t,t+2)$.

In a similar way, there exist an LOA$(s^{t-1},t, s, t-1)$ and an LOA$(2^{n+1},k_2',2,3)$
by Lemma \ref{5,6} and Theorem \ref{LOA-level2-str3} respectively.
By Construction \ref{3-2}, there exists an OA$(2^{n+1}{s^{t-1}}h',2^{k_1'}s^t,t+3)$.
\end{proof}

\begin{thm}\label{qtp43}
Suppose that $t$ is an integer, $p$ and $q$ are prime power with $p,q\geq3$, there exists an OA$({p^4}{q^t}{h},p^{k_1}q^{k_2},t+4)$ for $4\leq k_1\leq{p^2+1}$ and $2\leq t\leq {k_2}\leq q+1$, where $h=lcm\{p^{k_1-4},q^{k_2-t}\}$.
\end{thm}
\begin{proof}
For any prime power $p\geq3$ and $4\leq k_1\leq{p^2+1}$,
there exists an LOA$(p^4,k_1,p,3)$ by Theorem \ref{LOAq4-3}, and
for any prime power $q\geq3$ and $2\leq t\leq {k_2}\leq q+1$, there exists an LOA$(q^t,{k_2},q,t)$ by Lemma \ref{03}.
By Construction \ref{3-2}, there exists an OA$({p^4}{q^t}{h},p^{k_1}q^{k_2},6)$, where $h=lcm\{p^{k_1-4},q^{k_2-t}\}$.
\end{proof}

\section{Concluding remarks}

In this paper, we mainly focus on the construction of orthogonal arrays.
In view of Constructions \ref{3-1} and \ref{3-2}, a large set of orthogonal arrays
plays a crucial role on the construction of orthogonal arrays. Chen and Niu \cite{chen23} have obtained some infinite classes of
large sets of orthogonal arrays. Here we also obtain some new series of large set of orthogonal arrays, which are summarized in the following table.

\begin{center}
\textbf{Table 5 \ New LOAs }
\vskip 8pt
{\footnotesize
\begin{tabular}{lll}
\hline Parameters &  Constrains & References\\
\hline
LOA$(2^n,k,2,2)$&  $n\geq2$ and $n\leq k\leq 2^n-1$& Theorem \ref{LOA$(2^n)$}\\
LOA$(2^{n+1},k,2,3)$ & $n\geq2$ and $n+1\leq k\leq 2^n$& Theorem \ref{LOA-level2-str3}\\
LOA$(v^{3},k,v,2)$&$v\geq4$ and $v\not\equiv2 \pmod 4$, $4\leq k\leq 13$&Theorem \ref{chai11}\\
LOA$(v^{4},k,v,2)$ &$v\geq4$ and $v\not\equiv2 \pmod 4$, $4\leq k\leq29$&Theorem \ref{chai11}\\
LOA$(q^4,k,q,3)$& any prime power $q\geq 3$ and $4\leq k\leq q^2+1$& Theorem \ref{LOAq4-3}\\
\hline
\end{tabular}}
\end{center}
So there are about $16$ infinite classes large sets of orthogonal arrays from \textbf{Table 5} and Lemmas \ref{M-1}-\ref{5,6}, then we can obtain a new orthogonal array by combining any two of these large sets by Constructions \ref{3-1} and \ref{3-2}. Here we only list some orthogonal arrays of detailed parameters, see \textbf{Table 6}.

\begin{center}
\textbf{Table 6 \ \ \ New orthogonal arrays}
\vskip 8pt
{\scriptsize
\begin{tabular}{ccc}
\toprule[1pt] Parameters &  Constraints &References \\
\toprule[1pt]
\multirow{2}{*}{OA$(q^{k_2+1},2k_2-3,q,5)$}& any prime power $q\geq3$&$v=q$ and $k_1=k_2-1$\\
& and $4\leq k_2\leq{q^2+1}$&in Theorem \ref{v1+q4-3}\\\hline
\multirow{2}{*}{OA$(2^{k_2+1},2k_2-n+1,2,4)$}& $n\geq2$ and & $v=2$, $k_1=k_2-n+3$ \\
&  $n\leq k_2\leq{2^n-1}$& in Theorem \ref{v12n-4} \\\hline
\multirow{2}{*}{OA$(2^{k_2'+1},2k_2'-n,2,5)$}& $n\geq2$ and & $v=2$, $k_1=k_2'-n+2$\\
& ${n+1}\leq k_2'\leq{2^n}$&in Theorem \ref{v12n-4}\\\hline
\multirow{2}{*}{OA$(2^{k+n},2k,2,5)$}&$n\geq2$ and &$q=2$, $k_1=k_2=k$, $m=n$\\
&$n\leq k\leq{2^n-1}$&in Theorem \ref{qn2n-com}\\\hline
\multirow{2}{*}{OA$(2^{k_1+n},2k_1+1,2,6)$}&$n\geq2$ and &$q=2$, $k_1=k_2'-1$, $m=n$\\
&$n\leq k_1\leq{2^n-1}$&in Theorem \ref{qn2n-com}\\\hline
\multirow{2}{*}{OA$(q^{k_2+4},2k_2+2,q,6)$}&any prime power $q\geq3$,&$p=q$, $k_1=k_2+2$, $n=2$\\
&$4\leq k_2\leq{q+1}$&in Theorem \ref{qn2q43=6}\\\hline
\multirow{2}{*}{OA$(q^{k+4},2k,q,6)$}&any prime power $q\geq3$,&$p=q$, $k_1=k_2=k$, $n=4$\\
&$4\leq k\leq{q^2+1}$&in Theorem \ref{qn2q43=6}\\\hline
\multirow{2}{*}{OA$(2^{k_1m+n},(2^m)^{k_1}2^{(k_1-3)m+n},6)$}&$m\geq 1$, $n\geq 2$ and & $q=2^m$, $k_2=m(k_1-3)+n$\\
&$4\leq k_1\leq 2^m+2$, & in Theorem \ref{q3323=7}\\\hline
\multirow{2}{*}{OA$(2^{k_1m+n+1},(2^m)^{k_1}2^{(k_1-3)m+n+1},7)$}&$m\geq 1$, $n\geq 2$ and& $q=2^m$, $k_2=(k_1-3)m+n+1$\\
&$4\leq k_1\leq 2^m+2$, & in Theorem \ref{q3323=7}\\\hline
\multirow{2}{*}{OA$(q^{(k-4)t+4},(q^{k-4})^tq^k,t+3)$}& any prime power $q\geq 3$,& $s=q^{k-4}$\\
&any integer $t\geq 2$, $4\leq k\leq q^2+1$& in Theorem \ref{t-1q43=t+3}\\\hline
\multirow{2}{*}{OA$(2^{m{k_1}+n},{(2^m)}^{k_1}2^{m(k_1-t)+n},t+3)$}& $ m\geq 1$ and $n,t\geq 2$, & $q=2^m, k_2=m(k_1-t)+n$\\
& $2\leq k_1\leq 2^m+1$& in Theorem \ref{qt2n2-3}\\\hline
\multirow{2}{*}{OA$(2^{m{k_1}+n+1},{(2^m)}^{k_1}2^{m(k_1-t)+n+1},t+4)$}& $ m\geq 3$ and $n,t\geq 2$, & $q=2^m, {k_2}'=m(k_1-t)+n+1$\\
& $7\leq k_1\leq 2^m+1$& in Theorem \ref{qt2n2-3}\\\hline
\multirow{2}{*}{OA$(2^{({k_1}-n)t+n},{(2^{k_1-n})}^t2^{k_1},t+2)$}& $ t\geq 2$ and $n\geq 2$, & $s=2^{k_1-n}$\\
& $n\leq k_1\leq 2^n-1$& in Theorem \ref{tt-1n2-3}\\\hline
\multirow{2}{*}{OA$(2^{({k_1'}-n-1)t+n+1},{(2^{k_1'-n-1})}^t2^{k_1'},t+3)$}& $ t\geq 2$ and $n\geq 2$, & $s=2^{k_1'-n-1}$\\
& $n+1\leq k_1'\leq 2^n$& in Theorem \ref{tt-1n2-3}\\\hline
\multirow{2}{*}{OA$(q^{{k_1}+t},q^{2k_1+t-4},t+4)$}& prime power $q\geq 3$, $ q+1\geq t\geq 2$  & $p=q, k_2=k_1+t-4$\\
& $4\leq k_1\leq q+5-t$& in Theorem \ref{qtp43}\\
%\\\hline
\bottomrule[1pt]
\end{tabular}}
\end{center}

The saturation of the orthogonal arrays obtained by our constructions  is low, but
Constructions \ref{3-1}-\ref{3-2} give the efficient methods for construction of orthogonal arrays. How to give a new construction of orthogonal arrays with higher saturation will become the next important research problem and challenge.
The definition of an LOA with the same levels was proposed firstly by Stinson\cite{D.R} and it is used to computer science and cryptography. In addition, Zhang and Lei \cite{Z.Y} found that large sets of orthogonal arrays can be also used to construct multimagic squares. We believe that the results on LOA especially with mixed level have more
flexible and important applications. So exploring new applications of LOA is also a worthwhile research topic.

%\vskip 12pt
\newpage
\noindent
\textbf{Acknowledgments}

This paper was supported by the National Natural Science Foundation
of China under Grant No 11971104.

\newpage
\begin{center}
\textbf{Appendix A}
\end{center}
\vskip12pt

Note that the symbol ``T" represents the transposition of an array.
\vskip12pt

\ \ \ \ \ OA$(20,2^85^1,2)$  \ \ \ \ \ \ \ \ \ \ \ \ \ \ \ OA$(24,2^{13} 3^1 4^1,2)$ \ \ \ \ \ \ \ \ \ \  \ \ \ \ \ OA$(28,2^{12} 7^1,2)$

$\left(
\right)^T$

\end{document}